\documentclass[12pt]{article}

\usepackage{amsfonts}
\usepackage{amsmath}
\usepackage{amssymb}
\usepackage{amsthm}
\usepackage{enumitem}
\usepackage[pdfstartview=FitV,bookmarks=false]{hyperref}

\DeclareMathOperator{\Img}{Im}
\DeclareMathOperator{\Log}{Log}
\DeclareMathOperator{\ef}{e}

\newcommand{\dly}{\tau}
\newcommand{\cf}{p}
\newcommand{\ftrm}{q}
\newcommand{\initfun}{\varphi}
\newcommand{\prinfun}{\mathcal{X}}
\newcommand{\initpa}{\alpha}
\newcommand{\initpb}{\beta}
\newcommand{\initpc}{\gamma}
\newcommand{\initpd}{\zeta}
\newcommand{\seg}{\mathcal{I}}
\newcommand{\coseg}{\mathcal{J}}

\newcommand{\C}{\mathbb{C}}
\newcommand{\N}{\mathbb{N}}
\newcommand{\R}{\mathbb{R}}
\newcommand{\T}{\mathbb{T}}
\newcommand{\Z}{\mathbb{Z}}
\newcommand{\X}{\mathbb{X}}
\newcommand{\crd}{\mathrm{C}_{\mathrm{rd}}}
\newcommand{\reg}{\mathcal{R}}

\newtheorem{theorem}{Theorem}[section]

\newtheorem{definition}{Definition}[section]
\newtheorem{example}{Example}[section]

\newtheorem{remark}{Remark}[section]

\begin{document}
\title{Existence and uniqueness of solutions to systems of delay dynamic equations on time scales}
\author{Ba\c{s}ak KARPUZ\thanks{\textbf{Address}: Afyon Kocatepe University, Department of Mathematics, Faculty of Science and Arts, ANS Campus, 03200 Afyonkarahisar, Turkey.\newline\textbf{Email}: \texttt{bkarpuz@gmail.com}\newline\textbf{Web}: \url{http://www2.aku.edu.tr/\string~bkarpuz} }}
\date{\small{\textbf{2000 AMS Subject Classification}: Primary: 39A10, Secondary: 34C10.}}

\maketitle

\begin{abstract}
The purpose of this paper is to establish Picard-Lindel\"{o}f theorem for local uniqueness and existence results for first-order systems of nonlinear delay dynamic equations.
In the linear case, we extend our results to global existence and uniqueness of solutions on the entire interval (allowed to be unbounded above), and prove the variation of parameters formula for the unique solution in terms of the principal solution.
A simple example concerning the ordinary case is also provided.
\end{abstract}

\section{Introduction}\label{intro}

In \cite{hi1988}, the theory of time scale calculus is introduced by S.~Hilger in order to unify the discrete and the continuous analysis.
This theory allows the researchers to study dynamic equations, which include both difference and differential equations.
In the past decade, many papers appeared focusing on the investigation of the asymptotic and oscillatory properties of solutions of delay dynamic equations of the form
\begin{equation}
\begin{cases}
x^{\Delta}(t)=\displaystyle\sum_{i\in[1,n]_{\N}}\cf_{i}(t)x(\dly_{i}(t))+\ftrm(t)\quad\text{for}\ t\in[\initpb,\infty)_{\T}\\
x(t)=\initfun(t)\quad\text{for}\ t\in[\initpa,\initpb]_{\T},
\end{cases}\label{introeq1}
\end{equation}
where $n\in\N$, $\T$ is a time scale with $\sup\T=\infty$, $\initpa,\initpb\in\T$ with $\initpa\leq\min_{i\in[1,n]_{\N}}\inf_{t\in[\initpb,\infty)_{\T}}\{\dly_{i}(t)\}$, $\cf_{i},\ftrm\in\crd([\initpb,\infty)_{\T},\R)$, $\dly_{i}\in\crd([\initpb,\infty)_{\T},\T)$ satisfies $\lim_{t\to\infty}\dly_{i}(t)=\infty$ and $\dly_{i}(t)\leq t$ for all $t\in[\initpb,\infty)_{\T}$ and all $i\in[1,n]_{\N}$, and $\initfun\in\crd([\initpa,\initpb]_{\T},\R)$, for instance, see the papers \cite{ag2008,bo/pe2005,bo/ka/oc2008,sa/st2006,zh/de2002}.
All the mentioned papers assume existence of solutions to \eqref{introeq1}, however we have not met any result guaranteeing the existence and/or uniqueness of solutions for this type of equations in the literature yet.
In this paper, to fill the mentioned gap, we will prove Picard-Lindel\"{o}f existence and uniqueness theorem for the following nonlinear delay dynamic equation
\begin{equation}
\begin{cases}
x^{\Delta}(t)=f\big(t,x(\dly_{1}(t)),x(\dly_{2}(t)),\ldots,x(\dly_{n}(t))\big)\quad\text{for}\ t\in[\initpb,\initpc]^{\kappa}_{\T}\\
x(t)=\initfun(t)\quad\text{for}\ t\in[\initpa,\initpb]_{\T},
\end{cases}\label{introeq2}
\end{equation}
where $n\in\N$, $\T$ is a time scale, $\X$ is a Banach space, $\initpa,\initpb,\initpc\in\T$ with $\initpc\geq\initpb$ and $\initpa\leq\min_{i\in[1,n]_{\N}}\inf_{t\in[\initpb,\initpc]^{\kappa}_{\T}}\{\dly_{i}(t)\}$, $f\in\crd([\initpb,\initpc]^{\kappa}_{\T}\times\X^{n},\X)$, $\dly_{i}\in\crd([\initpb,\initpc]^{\kappa}_{\T},\T)$ satisfies $\dly_{i}(t)\leq t$ for $t\in[\initpb,\initpc]_{\T}$ and all $i\in[1,n]_{\N}$, and $\initfun\in\crd([\initpa,\initpb]_{\T},\X)$.
Then, we extend our results to prove existence and uniqueness for global solutions of the linear initial value problem \eqref{introeq1}.
It is clear that letting $\tau_{i}(t)=t$ for all $t\in[\initpb,\initpc]_{\T}$ and all $i\in[1,n]_{\N}$, \eqref{introeq2} reduces to the ordinary dynamic equation
\begin{equation}
\begin{cases}
x^{\Delta}(t)=f\big(t,x(t)\big)\quad\text{for}\ t\in[\initpb,\initpc]^{\kappa}_{\T}\\
x(\initpb)=x_{0},
\end{cases}\label{introeq3}
\end{equation}
where $x_{0}\in\X$.
Some existence and uniqueness results for \eqref{introeq3} can be found in \cite[\S~8.2]{bo/pe2001}, \cite[\S~2]{la/si/ka1996} and in the papers \cite{di/ru/sa2009,ka1993}.
In the continuous case ($\T=\R$), the readers may find Picard-Lindel\"{o}f theorem for \eqref{introeq3} in \cite[Theorem~8.13]{pe/ke2003}.
It should be also mentioned that our results (see Corollary~\ref{seurcrl1} and Remark~\ref{seurrmk3} below) particularly salvages the one given in \cite[Theorem~1.1.1]{gy/la1991} for delay differential equations.

For the completeness in the paper, we find useful to remind some
basic concepts of the time scale theory as follows.
A \emph{time scale}, which inherits the standard
topology on $\R$, is a nonempty closed subset of reals.
Throughout the paper, the notation $\T$ will be used to denote time scales.
On a time scale $\T$, the \emph{forward jump operator},
the \emph{backward jump operator} and the \emph{graininess function}
are respectively defined by
\begin{equation}
\sigma(t):=\inf(t,\infty)_{\T},\ \rho(t):=\sup(-\infty,t)_{\T}
\quad\text{and}\quad \mu(t):=\sigma(t)-t\notag
\end{equation}
for $t\in\T$, where the intervals with the subscript $\T$ denote the intersection of the usual interval with the time scale.
A point $t\in\T$ is called \emph{right-dense} if $t<\sup\T$ and $\sigma(t)=t$, while is called \emph{left-dense} if $t>\inf\T$ and $\rho(t)=t$.
Also, a point $t\in\T$ is called \emph{right-scattered} if $\sigma(t)>t$, and is called \emph{left-scattered} if $\rho(t)<t$.
The \emph{Hilger derivative} (in short derivative) of a function $f:\T\to\R$ is defined by
\begin{equation}
f^{\Delta}(t):=
\begin{cases}
\dfrac{f\big(\sigma(t)\big)-f(t)}{\mu(t)},&\mu(t)>0\\
\lim\limits_{\substack{s\in\T\\ s\to t}}\dfrac{f(t)-f(s)}{\sigma(t)-s},&\mu(t)=0
\end{cases}\notag
\end{equation}
for $t\in\T^{\kappa}$ (provided that limit exists), and $\T^{\kappa}:=\T\backslash\{\sup\T\}$ if $\sup\T=\max\T$ and satisfies $\rho(\max\T)\neq\max\T$; otherwise, $\T^{\kappa}:=\T$.
A function $f$ is called \emph{rd-continuous} provided
that it is continuous at right-dense points in $\T$,
and has finite limit at left-dense points, and
the \emph{set of rd-continuous} are denoted by $\crd(\T,\R)$, its subset $\crd^{1}(\T,\R)$
involves functions whose derivative is also in $\crd(\T,\R)$.
For a function $f\in\crd(\T,\R)$, the \emph{Cauchy integral} is defined by
\begin{equation}
\int_{s}^{t}f(\eta)\Delta\eta=F(t)-F(s)\quad\text{for}\ s,t\in\T,\notag
\end{equation}
where $F\in\crd^{1}(\T,\R)$ is an anti-derivative of the function $f$ on $\T$.
A function $f\in\crd(\T,\R)$ is called \emph{regressive} if
$1+\mu(t)f(t)\neq0$ for all $t\in\T$,
and $f\in\crd(\T,\R)$ is called \emph{positively regressive}
if $1+\mu(t)f(t)>0$ for all $t\in\T$.
The \emph{set of regressive functions} and the \emph{set
of positively regressive functions}
are defined by $\reg(\T,\R)$ and $\reg^{+}(\T,\R)$, respectively,
and the \emph{set of negatively regressive functions} $\reg^{-}(\T,\R)$ is
defined similarly.
Let $f\in\reg(\T,\R)$,
then the \emph{generalized exponential function}
$\ef_{f}(\cdot,s)$ on a time scale $\T$ is defined
to be the unique solution of the initial value problem
\begin{equation}
\begin{cases}
x^{\Delta}(t)=f(t)x(t)\quad\text{for}\ t\in\T^{\kappa}\\
x(s)=1
\end{cases}\notag
\end{equation}
for some fixed $s\in\T$.
For $h>0$, set $\C_{h}:=\{z\in\C:\ z\neq-1/h\}$,
$\Z_{h}:=\{z\in\C:\ -\pi/h<\Img(z)\leq\pi/h\}$, and $\C_{0}:=\Z_{0}:=\C$.
For $h\geq0$, define the \emph{cylinder transformation} $\xi_{h}:\C_{h}\to\Z_{h}$ by
\begin{equation}
\xi_{h}(z):=
\begin{cases}
z,&h=0\\
\displaystyle\frac{1}{h}\Log(1+hz),&h>0
\end{cases}\notag
\end{equation}
for $z\in\C_{h}$ with $1+hz\neq0$.
If $f\in\reg([s,t]_{\T},\R)$, then the exponential
function can also be written in the form
\begin{equation}
\ef_{f}(t,s):=\exp\bigg\{\int_{s}^{t}\xi_{\mu(\eta)}(f(\eta))\Delta\eta\bigg\}\quad\text{for}\ s,t\in\T.\notag
\end{equation}
It is known that the exponential function
$\ef_{f}(\cdot,s)$ is strictly positive on $[s,\infty)_{\T}$
provided that $f\in\reg^{+}([s,\infty)_{\T},\R)$, while
$\ef_{f}(\cdot,s)$ alternates in sign at right-scattered
points of the interval $[s,\infty)_{\T}$ provided that
$f\in\reg^{-}([s,\infty)_{\T},\R)$. %\cite[Theorem~2.48]{bo/pe2001}
The readers are referred to the books
\cite{bo/pe2001,la/si/ka1996} for further interesting details in the
time scale theory.

The paper is organized as follows:
In \S~\ref{seur}, we state and prove Picard-Lindel\"{o}f local existence and uniqueness of theorem to first-order nonlinear delay dynamic equation of the form \eqref{introeq2};
in \S~\ref{mole}, we focus our attention to linear equations in order to prove global existence and uniqueness of solutions together with the solution representation formula.
At the end of the paper, we provide a simple example.

\section{Picard-Lindel\"{o}f theorem for nonlinear delay dynamic equations}\label{seur}

In this section, we prove Picard-Lindel\"{o}f theorem for the solutions of the
nonlinear initial value problem \eqref{introeq2}.
Just as in the continuous and the discrete case the solution of the initial value problem in \eqref{introeq2} is equivalent to that of the corresponding integral equation
\begin{equation}
x(t)=
\begin{cases}
\initfun(\initpb)+\displaystyle\int_{\initpb}^{t}f\big(\eta,x(\dly_{1}(\eta)),x(\dly_{2}(\eta)),\ldots,x(\dly_{n}(\eta))\big)\Delta\eta\quad\text{for}\ t\in[\initpb,\initpc]_{\T}\\
\initfun(t)\quad\text{for}\ t\in[\initpa,\initpb]_{\T}.
\end{cases}\label{seureq2}
\end{equation}
Actually, we will prove the existence and uniqueness of solutions to the delay $\Delta$-integral equation \eqref{seureq2}.
The key idea of the proof depends on constructing a sequence of so-called successive Picard approximations, which uniformly converges to the unique solution of \eqref{seureq2}.
For $x_{0}\in\X$ and $\varepsilon\in\R^{+}$, we introduce the notation $B(x_{0},\varepsilon):=\{x\in\X:\ \|x-x_{0}\|\leq\varepsilon\}$ and define $I(\initfun,\varepsilon):=\bigcup_{t\in[\initpa,\initpb]_{\T}}B(\initfun(t),\varepsilon)$.
The main result of the paper reads as follows.

\begin{theorem}[Picard-Lindel\"{o}f theorem]\label{seurthm1}
Assume that for some $\varepsilon\in\R^{+}$, $f\in\crd([\initpb,\initpc]_{\T}^{\kappa}\times I^{n}(\varphi,\varepsilon),\X)$, and that for some $M\in\R^{+}$, $\|f(t,u_{1},u_{2},\ldots,u_{n})\|\leq M$ for all $t\in[\initpb,\initpc]_{\T}^{\kappa}$ and all $u_{1},u_{2},\ldots,u_{n}\in I(\varphi,\varepsilon)$, and that for some $L\in\R^{+}$, $f$ satisfies the Lipschitz condition
\begin{equation}
\big\|f(t,u_{1},u_{2},\ldots,u_{n})-f(t,v_{1},v_{2},\ldots,v_{n})\big\|\leq L\sum_{i\in[1,n]_{\N}}\big\|u_{i}-v_{i}\big\|\label{seurthm1eq1}
\end{equation}
for all $t\in[\initpb,\initpc]_{\T}^{\kappa}$ and all $u_{1},u_{2},\ldots,u_{n},v_{1},v_{2},\ldots,v_{n}\in I(\varphi,\varepsilon)$.
Then, the initial value problem \eqref{introeq2} has a unique solution $x$ on the interval $[\initpa,\sigma(\initpd)]_{\T}\subset[\initpa,\initpc]_{\T}$, where  $\initpd:=\max[\initpb,\initpb+\delta]_{\T}$ and $\delta:=\min\{\initpc-\initpb,\varepsilon/M\}$.
\end{theorem}

\begin{proof}
If $\initpb=\initpc$, then trivially the unique solution to \eqref{introeq2} solution is $\initfun$, we therefore consider the case $\initpc>\initpb$ below.
We shall first prove the existence of a unique solution on $[\initpa,\initpd]_{\T}$, and then extend to $[\initpa,\sigma(\initpd)]_{\T}$.
Define the space of functions $\Omega$ with the functions satisfying the following two properties:
\begin{enumerate}[label={(\roman*)},leftmargin=*,ref=(\roman*)]
\item $x\in\crd([\initpa,\initpb]_{\T},\X)$ and $x\in\crd^{1}([\initpb,\initpd]_{\T},\X)$,
\item $x(t)=\initfun(t)$ for all $t\in[\initpa,\initpb]_{\T}$ and  $x(t)\in B(\initfun(\initpb),\varepsilon)$ for all $t\in[\initpb,\initpd]_{\T}$.
\end{enumerate}
In view of \eqref{seureq2}, define the operator $\mathcal{T}:\Omega\to\Omega$ as follows:
\begin{equation}
(\mathcal{T} x)(t):=
\begin{cases}
\initfun(\initpb)+\displaystyle\int_{\initpb}^{t}f\big(\eta,x(\dly_{1}(\eta)),x(\dly_{2}(\eta)),\ldots,x(\dly_{n}(\eta))\big)\Delta\eta\quad\text{for}\ t\in[\initpb,\initpd]_{\T}\\
\initfun(t)\quad\text{for}\ t\in[\initpa,\initpb]_{\T}.
\end{cases}\label{seurthm1prfeq1}
\end{equation}
It is clear that $\mathcal{T}\Omega\subset\Omega$, and $\Omega$ is a closed subset of the Banach space $\crd([\initpa,\initpb]_{\T},\X)\cap\crd^{1}([\initpb,\initpd]_{\T},\X)$ endowed with the supremum norm.
Define the sequence of functions $\{x_{k}\}_{k\in\N}$ by
\begin{equation}
x_{k+1}(t):=(\mathcal{T}x_{k})(t)\quad\text{for}\ t\in[\initpa,\initpd]_{\T}\ \text{and}\ k\in\N,\label{seurthm1prfeq2}
\end{equation}
where
\begin{equation}
x_{0}(t):=
\begin{cases}
\initfun(\initpb)\quad\text{for}\ t\in[\initpb,\initpd]_{\T}\\
\initfun(t)\quad\text{for}\ t\in[\initpa,\initpb]_{\T}.
\end{cases}\notag
\end{equation}
By applying induction on $k\in\N$, we will prove that
\begin{equation}
\big\|x_{k}(t)-x_{k-1}(t)\big\|\leq ML^{k-1}h_{k}(t,\initpb)\quad\text{for all}\ t\in[\initpa,\initpd]_{\T}\ \text{and all}\ k\in\N,\label{seurthm1prfeq3}
\end{equation}
where $h_{k}:\T^{2}\to\R$ is the generalized polynomial defined by
\begin{equation}
h_{k}(t,s):=
\begin{cases}
1,&k=0\\
\displaystyle\int_{s}^{t}h_{k-1}(\eta,s)\Delta\eta,&k\in\N
\end{cases}\notag
\end{equation}
for $s,t\in\T$ (see \cite[\S~1.9]{bo/pe2001}).
By the definition of the sequence $\{x_{k}\}_{k\in\N}$, \eqref{seurthm1prfeq3} holds on $[\initpa,\initpb]_{\T}$.
Hence, it remains to prove \eqref{seurthm1prfeq3} on $[\initpb,\initpd]_{\T}$.
For all $t\in[\initpb,\initpd]_{\T}$, we have
\begin{align}
\big\|x_{1}(t)-x_{0}(t)\big\|=&\bigg\|\int_{\initpb}^{t}f\big(\eta,x_{0}(\dly_{1}(\eta)),x_{0}(\dly_{2}(\eta)),\ldots,x_{0}(\dly_{n}(\eta)\big)\Delta\eta\bigg\|\notag\\
\leq&M(t-\initpb)=Mh_{1}(t,\initpb),\notag
\end{align}
which proves \eqref{seurthm1prfeq3} for $k=1$.
Suppose now that \eqref{seurthm1prfeq3} is true for some $k\in\N$, then for all $t\in[\initpb,\initpd]_{\T}$ we have
\begin{align}
\begin{split}
\big\|x_{k+1}(t)-x_{k}(t)\big\|=&\bigg\|\int_{\initpb}^{t}\Big[f\big(\eta,x_{k}(\dly_{1}(\eta)),x_{k}(\dly_{2}(\eta)),\ldots,x_{k}(\dly_{n}(\eta)\big)\\
&-f\big(\eta,x_{k-1}(\dly_{1}(\eta)),x_{k-1}(\dly_{2}(\eta)),\ldots,x_{k-1}(\dly_{n}(\eta)\big)\Big]\Delta\eta\bigg\|
\end{split}\notag\\
\leq&L\int_{\initpb}^{t}\sum_{i\in[1,n]_{\N}}\big\|x_{k}(\dly_{i}(\eta))-x_{k-1}(\dly_{i}(\eta))\big\|\Delta\eta\notag\\
\leq&ML^{k}\int_{\initpb}^{t}h_{k}(\eta,\initpb)\Delta\eta=ML^{k}h_{k+1}(t,\initpb),\notag
\end{align}
which proves that \eqref{seurthm1prfeq3} is true when $k$ is replaced with $(k+1)$.
Hence, we have just proved \eqref{seurthm1prfeq3}.
Now, we prove that the sequence
\begin{equation}
\Bigg\{x_{0}(t)+\sum_{\ell\in[0,k)_{\N_{0}}}\big[x_{\ell+1}(t)-x_{\ell}(t)\big]\Bigg\}_{k\in\N}=\{x_{k}(t)\}_{k\in\N}\quad\text{for}\ t\in[\initpb,\initpd]_{\T}\label{seurthm1prfeq6}
\end{equation}
converges uniformly.
For all $t\in[\initpb,\initpd]_{\T}$ and all $k\in\N$, we have
\begin{align}
\big\|x_{k}(t)\big\|\leq&\big\|x_{0}(t)\big\|+\sum_{\ell\in[0,k)_{\N_{0}}}\big\|x_{\ell+1}(t)-x_{\ell}(t)\big\|\notag\\
\leq&\big\|x_{0}(t)\big\|+M\sum_{\ell\in[0,k)_{\N_{0}}}L^{\ell}h_{\ell+1}(t,\initpb),\notag
\end{align}
and for the majorant series, we have
\begin{align}
\big\|x_{0}(t)\big\|+M\sum_{\ell\in\N_{0}}L^{\ell}h_{\ell+1}(t,\initpb)=&\big\|\initfun(\initpb)\big\|+\frac{M}{L}\big(\ef_{L}(t,\initpb)-1\big)\notag\\
\leq&\big\|\initfun(\initpb)\big\|+\frac{M}{L}\big(\ef_{L}(\initpd,\initpb)-1\big)\notag
\end{align}
for any $t\in[\initpb,\initpd]_{\T}$ (see \cite[Theorem~1.113, Theorem~2.35]{bo/pe2001} and \cite[Lemma~4.4]{bo/gu07}).
Recall that for all $s,t\in\T$ with $t\geq s$, we have $\ef_{L}(t,s)\geq1$ from \cite[Theorem~1.76, Theorem~2.36, Exercise~2.46, Theorem~2.48]{bo/pe2001}.
It follows from the Weierstrass $M$-test that the infinite series
\begin{equation}
x_{0}(t)+\sum_{\ell\in\N_{0}}\big[x_{\ell+1}(t)-x_{\ell}(t)\big]\quad\text{for}\ t\in[\initpa,\initpd]_{\T},\notag
\end{equation}
whose partial sum sequence is \eqref{seurthm1prfeq6}, converges uniformly.
So that the sequence of Picard iterates $\{x_{k}\}_{k\in\N}$ converges uniformly on $[\initpa,\initpd]_{\T}$, let
\begin{equation}
x(t):=\lim_{k\to\infty}x_{k}(t)\quad\text{for}\ t\in[\initpa,\initpd]_{\T}.\label{seurthm1prfeq8}
\end{equation}
For all $t\in[\initpb,\initpd]_{\T}$ and all $k\in\N_{0}$, we have
\begin{align}
\big\|&f\big(t,x_{k}(\dly_{1}(t)),x_{k}(\dly_{2}(t)),\ldots,x_{k}(\dly_{n}(t))\big)\notag\\
&-f\big(t,x_{k}(\dly_{1}(t)),x_{k}(\dly_{2}(t)),\ldots,x_{k}(\dly_{n}(t))\big)\big\|\leq L\sum_{i\in[1,n]_{\N}}\big\|x_{k}(\dly_{i}(t))-x(\dly_{i}(t))\big\|.\notag
\end{align}
This proves that the sequence of functions $\{f(\cdot,x_{k}\circ\dly_{1},x_{k}\circ\dly_{2},\ldots,x_{k}\circ\dly_{n})\}_{k\in\N_{0}}$ converges uniformly to the limit function $f(\cdot,x\circ\dly_{1},x\circ\dly_{2},\ldots,x\circ\dly_{n})$ on $[\initpb,\initpd]_{\T}$.
Letting $k\to\infty$ in \eqref{seurthm1prfeq2} and using \cite[Theorem~3.11]{gu2003}, we learn that $x$ is the fixed point of the operator $\mathcal{T}$ defined by \eqref{seurthm1prfeq1}, i.e., $x=\mathcal{T}x$ on $[\initpa,\initpd]_{\T}$.
Therefore, we have just proved that $x$ defined by \eqref{seurthm1prfeq8} solves \eqref{seureq2} or equivalently \eqref{introeq2}, i.e., solutions to \eqref{introeq2} and/or \eqref{seureq2} exist.

To show uniqueness, assume for the sake of contradiction that $x$ and $y$ are two different solutions of \eqref{introeq2} on $[\initpa,\initpd]_{\T}$.
By the definition of the initial value problem $x=y$ on $[\initpa,\initpb]_{\T}$.
Set $z(t):=\sup_{\eta\in[\initpa,t]_{\T}}\|x(\eta)-y(\eta)\|$ for $t\in[\initpa,\initpd]_{\T}$, then $z\in\crd([\initpa,\initpd]_{\T},\R)$ is nonnegative and not identically zero.
Using \eqref{seureq2} and \eqref{seurthm1eq1}, we get
\begin{equation}
z(t)\leq nL\int_{\initpb}^{t}z(\eta)\Delta\eta\quad\text{for all}\ t\in[\initpb,\initpd]_{\T},\notag
\end{equation}
which yields by applying the well-known Gr\"{o}nwall inequality (see \cite[Theorem~6.4]{bo/pe2001}) that $z$ is nonpositive on $[\initpb,\initpd]_{\T}$.
This is a contradiction and thus the solution of \eqref{introeq2} is unique on $[\initpa,\initpd]_{\T}$.
The proof is completed here if $\initpd$ is right-dense.

Now let $\initpd$ be right-scattered, and $y$ be the solution of \eqref{introeq2} which exists uniquely on $[\initpa,\initpd]_{\T}$, and define $x:[\initpa,\sigma(\initpd)]_{\T}\to\X$ by
\begin{equation}
x(t):=
\begin{cases}
y(\initpd)+\mu(\initpd)f\big(\initpd,y(\dly_{1}(\initpd)),y(\dly_{2}(\initpd)),\ldots,y(\dly_{n}(\initpd))\big)\quad\text{for}\ t=\sigma(\initpd)\\
y(t)\quad\text{for}\ t\in[\initpa,\initpd]_{\T}.
\end{cases}\notag
\end{equation}
Then, $x$ is the unique solution of \eqref{introeq2} on $[\initpa,\sigma(\initpd)]_{\T}$, and this completes the proof.
\end{proof}

Note that for \eqref{introeq3}, the conclusion of Theorem~\ref{seurthm1} coincides with that of \cite[Theorem~2.1.1]{la/si/ka1996}.

\begin{remark}\label{seurrmk1}
Although Theorem~\ref{seurthm1} is a local existence result, we may interpret the solution as a new initial function, and apply the result repeatedly to obtain the unique solution on a larger interval.
\end{remark}

\begin{remark}\label{seurrmk2}
If the time scale $\T$ consists of only isolated points then the solution exists and is unique on the entire interval $[\initpa,\initpb]_{\T}$.
To see this, one may apply Theorem~\ref{seurthm1} for a total of several times (at most number of the elements in $[\initpa,\initpb]_{\T}$).
\end{remark}

\section{More on linear equations}\label{mole}

In this section, we restrict our attention to linear equations and extend the results in the previous section to prove existence and uniqueness of solutions on the whole interval, where the equation is defined.
Later on, we give the definition of the principal solution and prove the solution representation formula in terms of the principal solution and the initial function.

\subsection{Global existence and uniqueness of linear delay dynamic equations}\label{geau}

Let $\X$ be a Banach space and $B(\X)$ be the Banach algebra on $\X$.
In $B(\X)$, we now consider the following linear delay dynamic equation:
\begin{equation}
\begin{cases}
x^{\Delta}(t)=\displaystyle\sum_{i\in[1,n]_{\N}}\cf_{i}(t)x(\dly_{i}(t))+\ftrm(t)\quad\text{for}\ t\in[\initpb,\initpc]^{\kappa}_{\T}\\
x(t)=\initfun(t)\quad\text{for}\ t\in[\initpa,\initpb]_{\T},
\end{cases}\label{seureq4}
\end{equation}
where $[1,n]_{\N}$ is a starting bounded segment of $\N$, $\initpa,\initpb,\initpc\in\T$ with $\initpc>\initpb$ and $\initpa\leq\min_{i\in[1,n]_{\N}}\inf_{t\in[\initpb,\initpc]^{\kappa}_{\T}}\{\dly_{i}(t)\}$, $\cf_{i}\in\crd([\initpb,\initpc]^{\kappa}_{\T},B(\X))$, $\ftrm\in\crd([\initpb,\initpc]^{\kappa}_{\T},\X)$, $\dly_{i}\in\crd([\initpb,\initpc]^{\kappa}_{\T},\T)$ satisfies $\dly_{i}(t)\leq t$ for all $t\in[\initpb,\initpc]_{\T}$ and all $i\in[1,n]_{\N}$ and $\initfun\in\crd([\initpa,\initpb]_{\T},\X)$.

The main result of this subsection is the following.

\begin{theorem}\label{seurcrl1}
The linear initial value problem \eqref{seureq4} admits exactly one solution on the entire interval $[\initpa,\initpc]_{\T}$.
\end{theorem}

\begin{proof}
We shall apply the method of steps mentioned in Remark~\ref{seurrmk1} to obtain the unique solution on $[\initpa,\initpc]_{\T}$.
Pick $M_{1},M_{2}\in\R^{+}$ such that
\begin{equation}
\sup_{\eta\in[\initpb,\initpc]_{\T}}\sum_{i\in[1,n]_{\N}}\|\cf_{i}(\eta)\|\leq M_{1}\quad\text{and}\quad\sup_{\eta\in[\initpb,\initpc]_{\T}}\|\ftrm(\eta)\|\leq M_{2}\label{seurcrl1prfeq0}
\end{equation}
Recall that every rd-continuous (more truly regulated) function defined on a compact interval of $\T$ is bounded (see \cite[Theorem~1.65]{bo/pe2001}).
Define
\begin{equation}
f(t,u_{1},u_{2},\ldots,u_{n}):=\sum_{i\in[1,n]_{\N}}\cf_{i}(t)u_{i}+\ftrm(t)\notag
\end{equation}
for $t\in[\initpb,\initpc]_{\T}$ and $u_{1},u_{2},\ldots,u_{n}\in\X$.
Then, it is easy to see that the Lipschitz condition holds trivially on $[\initpb,\initpc]_{\T}\times\X^{n}$ (with the Lipschitz constant $M_{1}>0$).
For some fixed $k_{0}\in\N$, let $\initpb_{0}:=\initpb<\initpb_{1}<\cdots<\initpb_{k_{0}}:=\initpc$ be a partition of the interval $[\initpb,\initpc]_{\T}$ satisfying $\initpb_{k}=\sigma(\initpb_{k-1})$ if $\mu(\initpb_{k-1})>1/(2M_{1})$ and $|\initpb_{k}-\initpb_{k-1}|\leq1/(2M_{1})$ if $\mu(\initpb_{k-1})\leq1/(2M_{1})$ for all $k\in[1,k_{0}]_{\N}$ (see \cite[Lemma~2.6]{gu/ka2002}).
For convenience in the notation, we define $y_{0}:=\initfun$, $\initpb_{-1}:=\initpa$ and $\nu_{0}:=\sup_{\eta\in[\initpb_{-1},\initpb_{0}]_{\T}}\|y_{0}(\eta)\|$.
We may find $\varepsilon_{0}\in\R^{+}$ such that $\varepsilon_{0}/(M_{1}(\varepsilon_{0}+\nu_{0})+M_{2})\geq1/(2M_{1})$ (note that as $\varepsilon_{0}\to\infty$ the left-hand side of the inequality tends to $1/M_{1}$).
Clearly, $\|f(t,u_{1},u_{2},\ldots,u_{n})\|\leq M_{1}(\varepsilon_{0}+\nu_{0})+M_{2}$ for all $t\in[\initpb,\initpc]_{\T}^{\kappa}$ and all $u_{1},u_{2},\ldots,u_{n}\in B(0_{\X},\varepsilon_{0}+\nu_{0})$, where $0_{\X}$ is the zero vector in $\X$.
First consider the simple case $\mu(\initpb_{0})>1/(2M_{1})$, then we have $\sigma(\initpb_{0})=\initpb_{1}$.
Applying Theorem~\ref{seurthm1} to the initial value problem
\begin{equation}
\begin{cases}
x^{\Delta}(t)=\displaystyle\sum_{i\in[1,n]_{\N}}\cf_{i}(t)x(\dly_{i}(t))+\ftrm(t)\quad\text{for}\ t\in[\initpb_{0},\initpb_{1}]^{\kappa}_{\T}\\
x(t)=y_{0}(t)\quad\text{for}\ t\in[\initpb_{-1},\initpb_{0}]_{\T},
\end{cases}\label{seurcrl1prfeq1}
\end{equation}
we learn that \eqref{seurcrl1prfeq1} admits a unique solution $y_{1}$ on $[\initpb_{-1},\initpb_{1}]_{\T}$.
Next consider the case $\mu(\initpb_{0})\leq1/(2M_{1})$, then applying Theorem~\ref{seurthm1} to \eqref{seurcrl1prfeq1}, we again see that there exists a unique solution $y_{1}$ on $[\initpb_{-1},\initpb_{1}]_{\T}$.
It should be mentioned here that $I(y_{0},\varepsilon_{0})\subset B(0_{\X},\varepsilon_{0}+\nu_{0})$ and $\min\{\initpb_{1}-\initpb_{0},\varepsilon_{0}/(M_{1}(\varepsilon_{0}+\nu_{0})+M_{2})\}\geq\min\{\initpb_{1}-\initpb_{0},1/(2M_{1})\}$, and this shows that the width of the existence interval for the unique solution of \eqref{seurcrl1prfeq2} guaranteed by Theorem~\ref{seurthm1} can not be less than we have proved.
In the next step, we may pick $\varepsilon_{1}\in\R^{+}$ such that $\varepsilon_{1}/(M_{1}(\varepsilon_{1}+\nu_{1})+M_{2})\geq1/(2M_{1})$, where $\nu_{1}:=\sup_{\eta\in[\initpb_{-1},\initpb_{1}]_{\T}}\|y_{1}(\eta)\|$.
Then we have $\|f(t,u_{1},u_{2},\ldots,u_{n})\|\leq M_{1}(\varepsilon_{1}+\nu_{1})+M_{2}$ for all $t\in[\initpb,\initpc]_{\T}^{\kappa}$ and all $u_{1},u_{2},\ldots,u_{n}\in B(0_{\X},\varepsilon_{1}+\nu_{1})$.
It follows from Theorem~\ref{seurthm1} (in both of the possible cases $\mu(\initpb_{1})>1/(2M_{1})$ and $\mu(\initpb_{1})\leq1/(2M_{1})$ as shown in the first step) that the initial value problem
\begin{equation}
\begin{cases}
x^{\Delta}(t)=\displaystyle\sum_{i\in[1,n]_{\N}}\cf_{i}(t)x(\dly_{i}(t))+\ftrm(t)\quad\text{for}\ t\in[\initpb_{1},\initpb_{2}]^{\kappa}_{\T}\\
x(t)=y_{1}(t)\quad\text{for}\ t\in[\initpb_{-1},\initpb_{1}]_{\T}
\end{cases}\label{seurcrl1prfeq2}
\end{equation}
admits a unique solution $y_{2}$ on $[\initpb_{-1},\initpb_{2}]_{\T}$ since $I(y_{1},\varepsilon_{1})\subset B(0_{\X},\varepsilon_{1}+\nu_{1})$ and $\min\{\initpb_{2}-\initpb_{1},\varepsilon_{1}/(M_{1}(\varepsilon_{1}+\nu_{1})+M_{2})\}\geq\min\{\initpb_{2}-\initpb_{1},1/(2M_{1})\}$.
Repeating in this manner, we obtain the unique solution $y_{k}$ on each of the intervals $[\initpb_{-1},\initpb_{k}]_{\T}$ for all $k\in[1,k_{0}]_{\N}$.
Finally, we deduce that the unique solution $x$ of \eqref{seureq4} on the whole interval $[\initpa,\initpc]_{\T}$ is $x=y_{k_{0}}$, and the proof is therefore completed.
\end{proof}

\begin{remark}\label{seurrmk3}
The claim of Theorem~\ref{seurcrl1} remains true for \eqref{introeq2} provided that there exist constants $\cf_{1},\cf_{2},\ldots,\cf_{n},\ftrm\in\crd([\initpb,\initpc]_{\T},\R^{+})$ such that
\begin{equation}
\|f(t,u_{1},u_{2},\ldots,u_{n})\|\leq\sum_{i\in[1,n]_{\N}}\cf_{i}(t)\|u_{i}\|+\ftrm(t)\notag
\end{equation}
for all $t\in[\initpb,\initpc]_{\T}$ and all $u_{1},u_{2},\ldots,u_{n}\in\X$.
\end{remark}

\begin{remark}\label{seurrmk2}
Theorem~\ref{seurcrl1} can be used to prove existence and uniqueness of global solutions to the initial value problem \eqref{introeq1} by the method of steps.
More precisely, pick an increasing divergent sequence $\{\initpb_{k}\}_{k\in\N}\subset[\initpb,\infty)_{\T}$ such that $\initpb_{k-1}\leq\min_{i\in[1,n]_{\N}}\inf_{t\in[\initpb_{k},\infty)_{\T}}\{\dly_{i}(t)\}$ for all $k\in\N$ (with the convention $\initpb_{-1}:=\initpa$ and $\initpb_{0}:=\initpb$).
Successively for $k\in\N$, apply Theorem~\ref{seurcrl1} on $[\initpb_{k-2},\initpb_{k}]_{\T}$ by considering the unique solution obtained in the previous step as the initial function on $[\initpb_{k-2},\initpb_{k-1}]_{\T}$, and denote this solution by $y_{k}$.
Finally, we obtain the unique global solution $x$ of \eqref{introeq1} by letting $x(t)=y_{k}(t)$ for $t\in[\initpb_{k-1},\initpb_{k}]_{\T}$ and $k\in\N_{0}$.
\end{remark}

\subsection{Representation of solutions by means of the principal solution}\label{srf}

In this section, we give the definition of the principal solution and prove the representation formula for the unique solution of \eqref{seureq4} in terms of the principal solution.
We need to recall the definition of the characteristic function $\chi$.
Let $U\subset\R$, and define the so-called characteristic function $\chi_{U}:\R\to\{0,1\}$ to be
\begin{equation}
\chi_{U}(t):=
\begin{cases}
1,&t\in U\\
0,&\text{otherwise}
\end{cases}\notag
\end{equation}
for $t\in\R$.

Now we can give the definition of the principal solution to \eqref{seureq4}.

\begin{definition}[Principal solution]\label{predf1}
Let $\initpd\in[\initpb,\initpc]_{\T}$.
The solution $\prinfun=\prinfun(\cdot,\initpd):[\initpa,\initpc]_{\T}\to\X$ of the problem
\begin{equation}
\begin{cases}
x^{\Delta}(t)=\displaystyle\sum_{i\in[1,n]_{\N}}\cf_{i}(t)x(\dly_{i}(t))\quad\text{for}\ t\in[\initpd,\initpc]_{\T}^{\kappa}\\
x(t)=\chi_{\{\initpd\}}(t)1_{\X}\quad\text{for}\ t\in[\initpa,\initpd]_{\T},
\end{cases}\notag
\end{equation}
where $1_{\X}$ is the unity of the Banach algebra $B(\X)$, which satisfies $\prinfun(\cdot,\initpd)\in\crd^{1}([\initpd,\initpc]_{\T},\X)$, is called the \emph{principal solution} of \eqref{seureq4}.
\end{definition}

The following result, which plays the major role (both in the continuous and in the discrete cases)
in the qualitative theory by suggesting a representation formula
to solutions of \eqref{seureq4} by the means of
the principal solution $\prinfun$, is proven below.
This result is also known as the \textquotedblleft variation of parameters formula\textquotedblright.

\begin{theorem}[Solution representation formula]\label{prethm2}
Let $x$ be a solution of \eqref{seureq4},
then $x$ can be written in the following form:
\begin{equation}
\begin{split}
x(t)=&\prinfun(t,\initpb)\initfun(\initpb)+\int_{\initpb}^{t}\prinfun(t,\sigma(\eta))\ftrm(\eta)\Delta\eta\\
&+\int_{\initpb}^{t}\prinfun(t,\sigma(\eta))\sum_{i\in[1,n]_{\N}}\cf_{i}(\eta)\chi_{[\initpa,\initpb)_{\T}}(\dly_{i}(\eta))\initfun(\dly_{i}(\eta))\Delta\eta
\end{split}\label{prelm1eq1}
\end{equation}
for $t\in[\initpb,\initpc]_{\T}$.
\end{theorem}

\begin{proof}
In the light of Theorem~\ref{seurcrl1}, it suffices to prove that
\begin{equation}
y(t):=
\begin{cases}
\begin{aligned}
&\prinfun(t,\initpb)\initfun(\initpb)+\int_{\initpb}^{t}\prinfun(t,\sigma(\eta))\ftrm(\eta)\Delta\eta\\
&+\int_{\initpb}^{t}\prinfun(t,\sigma(\eta))\sum_{i\in[1,n]_{\N}}\cf_{i}(\eta)\chi_{[\initpa,\initpb)_{\T}}(\dly_{i}(\eta))\initfun(\dly_{i}(\eta))\Delta\eta
\end{aligned},&t\in[\initpb,\initpc]_{\T}\\
\initfun(t),&t\in[\initpa,\initpb]_{\T}
\end{cases}\notag
\end{equation}
solves \eqref{seureq4}.
For $t\in[\initpb,\initpc]_{\T}$, set $\seg(t)=\{j\in[1,n]_{\N}:\ \chi_{[\initpb,\initpc]_{\T}}(\dly_{j}(t))=1\}$ and $\coseg(t):=\{j\in[1,n]_{\N}:\ \chi_{[\initpa,\initpb)_{\T}}(\dly_{j}(t))=1\}$.
Considering the definition of the principal solution $\prinfun$, we have
\begin{align}
\begin{split}
y^{\Delta}(t)=&\prinfun^{\Delta}(t,\initpb)\initfun(\initpb)+\int_{\initpb}^{t}\prinfun^{\Delta}(t,\sigma(\eta))\ftrm(\eta)\Delta\eta+\prinfun(\sigma(t),\sigma(t))\ftrm(t)\\
&+\int_{\initpb}^{t}\prinfun^{\Delta}(t,\sigma(\eta))\sum_{i\in[1,n]_{\N}}\cf_{i}(\eta)\chi_{[\initpa,\initpb)_{\T}}(\dly_{i}(\eta))\initfun(\dly_{i}(\eta))\Delta\eta\\
&+\prinfun(\sigma(t),\sigma(t))\sum_{i\in[1,n]_{\N}}\cf_{i}(t)\chi_{[\initpa,\initpb)_{\T}}(\dly_{i}(t))\initfun(\dly_{i}(t))
\end{split}\notag\\
\begin{split}
=&\sum_{j\in\seg(t)}\cf_{j}(t)\Bigg[\prinfun(\dly_{j}(t),\initpb)\initfun(\initpb)+\int_{\initpb}^{t}\prinfun(\dly_{j}(t),\sigma(\eta))\ftrm(\eta)\Delta\eta\\
&+\int_{\initpb}^{t}\prinfun(\dly_{j}(t),\sigma(\eta))\sum_{i\in[1,n]_{\N}}\cf_{i}(\eta)\chi_{[\initpa,\initpb)_{\T}}(\dly_{i}(\eta))\initfun(\dly_{i}(\eta))\big]\Delta\eta\Bigg]\\
&+\sum_{j\in\coseg(t)}\cf_{j}(t)\initfun(\dly_{j}(t))+\ftrm(t)
\end{split}\notag
\end{align}
for all $t\in[\initpb,\initpc]_{\T}^{\kappa}$.
Making some arrangements, we get
\begin{align}
\begin{split}
y^{\Delta}(t)=&\sum_{j\in\seg(t)}\cf_{j}(t)\Bigg[\prinfun(\dly_{j}(t),\initpb)\initfun(\initpb)+\int_{\initpb}^{\dly_{j}(t)}\prinfun(\dly_{j}(t),\sigma(\eta))\ftrm(\eta)\Delta\eta\\
&+\int_{\initpb}^{\dly_{j}(t)}\prinfun(\dly_{j}(t),\sigma(\eta))\sum_{i\in[1,n]_{\N}}\cf_{i}(\eta)\chi_{[\initpa,\initpb)_{\T}}(\dly_{i}(\eta))\initfun(\dly_{i}(\eta))\big]\Delta\eta\Bigg]\\
&+\sum_{j\in\coseg(t)}\cf_{j}(t)\initfun(\dly_{j}(t))+\ftrm(t)
\end{split}\notag\\
=&\sum_{j\in\seg(t)}\cf_{j}(t)y(\dly_{j}(t))+\sum_{j\in\coseg(t)}\cf_{j}(t)y(\dly_{j}(t))+\ftrm(t),\notag
\end{align}
which proves that $y$ satisfies \eqref{seureq4} for all $t\in[\initpb,\initpc]_{\T}^{\kappa}$ since $\seg(t)\cap\coseg(t)=\emptyset$ and $\seg(t)\cup\coseg(t)=[1,n]_{\N}$ for each $t\in[\initpb,\initpc]_{\T}$.
The proof is therefore completed.
\end{proof}

We conclude the paper with the following nice and simple example, which considers first-order ordinary dynamic equations.

\begin{example}\label{preex1}
Consider the following scalar first-order dynamic equation:
\begin{equation}
\begin{cases}
x^{\Delta}(t)=\cf(t)x(t)+\ftrm(t)\quad\text{for}\ t\in[\initpb,\initpc]^{\kappa}_{\T}\\
x(\initpb)=x_{0}.
\end{cases}\label{preex1eq1}
\end{equation}
The principal solution of \eqref{preex1eq1} is $\prinfun(t,s)=\chi_{[s,\initpc]_{\T}}(t)\ef_{\cf}(t,s)$ for $s,t\in[\initpb,\initpc]_{\T}$ provided that $\cf\in\reg([\initpb,\initpc]_{\T},\R)$ (see \cite[Theorem~2.71]{bo/pe2001}).
Note here that the need for the regressivity condition is just to make the exponential function meaningful.
Thus, due to Theorem~\ref{prethm2} the unique solution can be written in the form
\begin{equation}
x(t)=x_{0}\ef_{\cf}(t,\initpb)+\int_{\initpb}^{t}\ef_{\cf}(t,\sigma(\eta))\ftrm(\eta)\Delta\eta\notag
\end{equation}
for $t\in[\initpb,\initpc]_{\T}$ (see \cite[Theorem~2.77]{bo/pe2001}).
\end{example}

\end{document}